\documentclass[11pt,a4paper]{article}
\usepackage[T1]{fontenc}
\usepackage[utf8]{inputenc}
\usepackage{lmodern}
\usepackage{microtype}
\usepackage{amsmath,amssymb,amsthm,mathtools}
\usepackage[margin=29mm]{geometry}
\usepackage{enumitem}
\usepackage{xcolor}
\usepackage[colorlinks=true,linkcolor=blue!45!black,citecolor=blue!45!black,urlcolor=blue!45!black]{hyperref}
\hypersetup{pdftitle={Negative Values of Wahl's Second Chern Number for Normal Surface Singularities},pdfsubject={Counterexamples to Wahl's vanishing conjecture},pdfauthor={}}

\numberwithin{equation}{section}
\newtheorem{theorem}{Theorem}[section]
\newtheorem{proposition}[theorem]{Proposition}
\newtheorem{lemma}[theorem]{Lemma}
\newtheorem{corollary}[theorem]{Corollary}
\theoremstyle{definition}

\theoremstyle{remark}
\newtheorem{remark}[theorem]{Remark}
\newcommand{\cw}{c_2^{\mathrm W}}
\newcommand{\crel}{c_1^{\mathrm{rel}}}
\newcommand{\etop}{\chi_{\mathrm{top}}}
\newcommand{\OO}{\mathcal O}
\newcommand{\FF}{\mathcal F}
\newcommand{\GG}{\mathcal G}
\newcommand{\QQ}{\mathcal Q}
\newcommand{\wahlresidueattribution}{Wahl \cite[Proposition~2.14]{Wahl}}
\DeclareMathOperator{\Sym}{Sym}
\DeclareMathOperator{\Pic}{Pic}
\DeclareMathOperator{\Spec}{Spec}

\DeclareMathOperator{\Fix}{Fix}

\setlist{topsep=4pt,itemsep=2pt}

\title{Negative Values of Wahl's Second Chern Number\\
for Normal Surface Singularities}
\author{Rong Du\thanks{The author is sponsored by National Natural Science Foundation of China (Grant No.~12471040) and Science and Technology Commission of Shanghai Municipality (Grant No.~22DZ2229014).}}
\date{}

\begin{document}
\maketitle

\setcounter{footnote}{0}
\vspace{-1.7em}
\begin{abstract}
We disprove Wahl's 1993 vanishing conjecture, open for more than thirty
years, which asserts that $\cw(X)=0$ for every nonquotient normal surface
singularity.
Let $C$ be a smooth projective curve with a
nontrivial involution having $r>0$ fixed points, and put $d=2g(C)-2$.
If $d>r$, the union of the diagonal and the graph of the involution in
$C\times C$ has negative definite intersection matrix and can be
analytically contracted to a normal surface singularity $(X,o)$.
A direct relative Chern-class computation and Wahl's residue formula give
$\cw(X)=-r(d+r)/(d-r)$.
An explicit curve of genus four gives $\cw(X)=-4$. Branched double covers
of a fixed curve of genus two yield examples with
$\cw(X_m)=-2m(m+1)$, so these values are unbounded below.
We give the construction and intersection calculations in detail, as
well as a line-bundle filtration realizing the negative value.
\end{abstract}


\section{Introduction}

Let $\pi:(U,E)\to(X,o)$ be a good resolution of a complex normal surface
singularity. The reduced exceptional divisor $E$ is a connected union
of smooth compact curves with only transverse double intersections.
The intersection matrix of its irreducible components is negative
definite. This matrix allows one to associate relative first Chern
classes, and hence rational relative intersection numbers, to line
bundles on $U$.

Defining a corresponding second Chern number for a vector bundle is a
different problem. On a sufficiently small resolution neighborhood,
the ordinary absolute second Chern class contains no such information.
Wahl~\cite{Wahl} introduced a relative second Chern number for rank-two
bundles by taking an infimum of intersection numbers of line-bundle
filtrations after proper generically finite pullback. His construction
was motivated by a local asymptotic Riemann--Roch formula for symmetric
powers. Langer subsequently extended the construction to higher rank
and studied its relation to Chern classes of reflexive sheaves and
local Riemann--Roch corrections~\cite{Langer}.

For the logarithmic cotangent bundle, Wahl proved that
\begin{equation}\label{eq:singularity-invariant}
 \cw(X):=\cw\bigl(\Omega_U^1(\log E)\bigr)
\end{equation}
is independent of the good resolution
\cite[Definition--Proposition~2.9]{Wahl}. Throughout this paper the
superscript $\mathrm W$ specifies Wahl's definition.

Several computations in \cite{Wahl} suggested a vanishing phenomenon.
For a weighted homogeneous normal surface singularity, Wahl obtained
\begin{equation}\label{eq:weighted-background}
 \cw(X)=-\frac{1}{|\pi_1(\partial X)|},
\end{equation}
where the right-hand side is interpreted as zero if the fundamental
group of the link is infinite
\cite[Theorem~2.11]{Wahl}. In particular, it vanishes for the nonquotient
weighted homogeneous singularities. Wahl also proved vanishing for
cusp singularities \cite[Proposition~2.13]{Wahl}. These results led to
the following statement:
\begin{quote}
\textbf{Wahl's Conjecture} \textnormal{(\cite[Conjecture~2.12]{Wahl}).}
One has \mbox{$\cw(X)=0$} for every nonquotient normal surface
singularity $(X,o)$.
\end{quote}

The conjecture predicts a sharp numerical dichotomy: Wahl's
resolution-independent second Chern number should be nonzero precisely
for quotient singularities and should vanish for every nonquotient
singularity. Since this invariant was designed to supply the second
Chern term in a local asymptotic Riemann--Roch theory, the conjecture
also links the birational type of a singularity to the leading
cohomological growth of symmetric powers. Formulated in 1993, it has
stood for more than three decades as a basic test of the theory.

We show that this proposed dichotomy fails strongly in the original
complex analytic setting: nonquotient singularities can have strictly
negative Wahl second Chern number, and these values are unbounded below.
The key observation is that an exceptional divisor can be embedded in
a product of curves, where the ordinary cotangent bundle splits.
Wahl's comparison between ordinary and logarithmic cotangent bundles
then makes the desired invariant computable.

\begin{theorem}\label{thm:main}
Let $C$ be a smooth connected projective complex curve, let $\sigma$ be
a nontrivial involution of $C$, and suppose that
\[
 r=\#\Fix(\sigma)>0,\qquad d=2g(C)-2>r.
\]
In $S=C\times C$, let
\[
 E_1=\Delta_C:=\{(p,p)\mid p\in C\},\qquad
 E_2=\Gamma_\sigma:=\{(p,\sigma(p))\mid p\in C\}
\]
be the diagonal and the graph of $\sigma$, respectively.
The connected divisor $E=E_1+E_2$ can be analytically contracted to a
normal surface singularity $(X,o)$, which is not a quotient singularity.
For this contraction,
\begin{equation}\label{eq:main-formula}
 \cw(X)=-\frac{r(d+r)}{d-r}<0.
\end{equation}
\end{theorem}

The theorem will be proved in Sections~\ref{sec:geometry}
and~\ref{sec:calculation}. Section~\ref{sec:explicit} provides the
explicit example
\[
 C:\quad y^2=\prod_{j=1}^{5}(x^2-j),\qquad
 \sigma(x,y)=(-x,y),
\]
for which $g(C)=4$, $r=2$, and $\cw(X)=-4$.
Section~\ref{sec:families} constructs the following unbounded family.

\begin{corollary}\label{cor:unbounded-intro}
For every integer $m\ge1$, there is a nonquotient complex analytic normal
surface singularity $(X_m,o_m)$ such that
\begin{equation}\label{eq:unbounded-intro}
 \cw(X_m)=-2m(m+1).
\end{equation}
Consequently, Wahl's second Chern number is unbounded below among
nonquotient complex analytic normal surface singularities.
\end{corollary}

The vanishing conjecture should be distinguished from Wahl's main
asymptotic Riemann--Roch conjecture
\cite[Main Conjecture~3.3]{Wahl}. Our calculation does not assume that
formula. In fact, the formula holds for the particular ordinary and
logarithmic cotangent bundles used here; we explain this in
Section~\ref{sec:rr}.

\section{Relative intersection theory and Wahl's formulas}
\label{sec:background}

\subsection{Relative first Chern classes}

Write $E=\bigcup_{i=1}^s E_i$ and let
$Q=(E_i\cdot E_j)_{1\le i,j\le s}$.
For a line bundle $A$ on $U$, there is a unique rational exceptional
cycle
\begin{equation}\label{eq:relative-class}
 \crel(A)=\sum_{i=1}^s a_iE_i
 \quad\text{such that}\quad
 \crel(A)\cdot E_j=\deg(A|_{E_j})\quad\text{for every }j.
\end{equation}
Existence and uniqueness follow from the invertibility of the negative
definite matrix $Q$.
If $\boldsymbol{\ell}_A$ denotes the column vector of these degrees,
then the coefficient vector of $\crel(A)$ is
$Q^{-1}\boldsymbol{\ell}_A$.
For two line bundles $A$ and $B$, define
\begin{equation}\label{eq:relative-pairing}
 A\cdot B:=\crel(A)\cdot\crel(B)
 =\boldsymbol{\ell}_A^{\mathsf T}Q^{-1}\boldsymbol{\ell}_B.
\end{equation}
For a vector bundle $\FF$, put
$\crel(\FF)=\crel(\det\FF)$.

An intersection with a compact exceptional curve, such as
$A\cdot E_i=\deg(A|_{E_i})$, is the usual degree.
However, the pairing of two line bundles in
\eqref{eq:relative-pairing} is a \emph{relative} pairing.
Even when $U$ is an open subset of a compact surface $S$ and the line
bundles extend to $S$, it is generally different from their global
intersection on $S$. This distinction will be essential below.

\subsection{The relative second Chern number}

We recall the part of Wahl's definition needed here
\cite[\S2]{Wahl}. For a rank-two vector bundle $\FF$ on $U$, consider
proper generically finite maps $f:V\to U$ in the resolution setting,
with $V$ smooth, and exact sequences
\begin{equation}\label{eq:admissible-filtration}
 0\longrightarrow A\longrightarrow f^*\FF
 \longrightarrow B\longrightarrow0,
\end{equation}
where $A$ and $B$ are line bundles. The invariant is
\begin{equation}\label{eq:wahl-definition}
 \cw(\FF)=\inf_{f,A,B}\frac{A\cdot B}{\deg f}.
\end{equation}
The maps in this definition are not required to be unramified away
from the exceptional set. No reduction of this infimum to unramified
covers is used in this paper.

Negative definiteness gives a useful universal lower bound.
For any filtration \eqref{eq:admissible-filtration}, write
$a=\crel(A)$ and $b=\crel(B)$. Then
\[
 a\cdot b=\frac14(a+b)^2-\frac14(a-b)^2
 \ge\frac14(a+b)^2.
\]
Relative intersections of pullback classes multiply by $\deg f$.
Dividing by this degree and taking the infimum yields
\begin{equation}\label{eq:discriminant-bound}
 \cw(\FF)\ge\frac14\crel(\FF)^2,
\end{equation}
This is Wahl's inequality \cite[p.~88, (2.1.1)]{Wahl}.

Wahl also proved
\begin{equation}\label{eq:split-property}
 \cw(A\oplus B)=A\cdot B
\end{equation}
for arbitrary line bundles $A,B$
\cite[Proposition~2.5]{Wahl}. The special case needed in our construction
has an especially short proof directly from the definition.

\begin{lemma}\label{lem:equal-classes}
If $\FF=A\oplus B$ and $\crel(A)=\crel(B)=a$, then
\[
 \cw(\FF)=a^2.
\]
\end{lemma}
\begin{proof}
The splitting on the identity cover gives
$\cw(\FF)\le A\cdot B=a^2$.
Since $\crel(\FF)=2a$, the opposite inequality follows from
\eqref{eq:discriminant-bound}.
\end{proof}

\subsection{Passing from ordinary to logarithmic differentials}

Logarithmic differential sheaves on resolutions of surface singularities
have also been considered in connection with plurigenus invariants; in
particular, they occur in the author's joint work with Yau on the second plurigenus of
smoothable Gorenstein surface singularities \cite{DuYau}.

For a reduced simple normal crossings divisor $E$, the residue sequence is
\begin{equation}\label{eq:residue-sequence}
 0\longrightarrow\Omega_U^1
 \longrightarrow\Omega_U^1(\log E)
 \longrightarrow\bigoplus_{i=1}^s\OO_{E_i}
 \longrightarrow0.
\end{equation}
The determinant of the middle term is $\OO_U(K_U+E)$.

\begin{proposition}[\wahlresidueattribution]\label{prop:residue-identity}
For every good resolution $(U,E)\to(X,o)$,
\begin{equation}\label{eq:residue-identity}
 \cw(\Omega_U^1)
 =\cw\bigl(\Omega_U^1(\log E)\bigr)+\etop(E).
\end{equation}
\end{proposition}
\begin{proof}
Wahl's change-of-bundle formula
\cite[Theorem~2.7]{Wahl} states that if
\[
 0\longrightarrow\FF\longrightarrow\GG
 \longrightarrow\QQ\longrightarrow0
\]
is an exact sequence of rank-two bundles with exceptional quotient,
and $Z=\det\GG-\det\FF$, then
\begin{equation}\label{eq:change-bundle}
 \cw(\FF)=\cw(\GG)-\det(\GG)\cdot Z
 +\frac12Z\cdot(Z+K_U)+\chi(\QQ).
\end{equation}
Applied to \eqref{eq:residue-sequence}, this becomes
\begin{equation}\label{eq:residue-intermediate}
 \cw(\Omega_U^1)-\cw\bigl(\Omega_U^1(\log E)\bigr)
 =-\frac12(K_U+E)\cdot E+\sum_{i=1}^s(1-g_i),
\end{equation}
where $g_i=g(E_i)$.

Let $t$ be the total number of nodes of $E$, and let $v_i$ be the number
of these nodes lying on $E_i$. By adjunction,
\[
 (K_U+E)\cdot E_i=2g_i-2+v_i,
 \qquad \sum_i v_i=2t.
\]
The right-hand side of \eqref{eq:residue-intermediate} therefore equals
\[
 -\frac12\left(\sum_i(2g_i-2)+2t\right)
 +\sum_i(1-g_i)
 =\sum_i(2-2g_i)-t.
\]
This is exactly $\etop(E)$.
\end{proof}

All appearances of $\chi(\QQ)$ above refer to the coherent Euler
characteristic of a sheaf with compact support. In contrast,
$\etop(E)$ denotes the topological Euler characteristic of the reduced
exceptional curve.

\section{A contractible union of two graphs}\label{sec:geometry}

Let $C$ be a smooth connected projective complex curve and let
$\sigma:C\to C$ be a nontrivial involution. Set
\[
 r:=\#\Fix(\sigma)>0,\qquad d:=2g(C)-2,
\]
and assume $d>r$. On $S=C\times C$, put
\[
 E_1:=\Delta_C=\{(p,p)\mid p\in C\},\qquad
 E_2:=\Gamma_\sigma=\{(p,\sigma(p))\mid p\in C\},
 \qquad E:=E_1+E_2.
\]
We first verify every geometric condition needed for the contraction.

\begin{lemma}\label{lem:graph-normal}
For an automorphism $\tau$ of $C$, the normal bundle of its graph
$\Gamma_\tau\subset C\times C$ is isomorphic to $\tau^*T_C$, and hence
to $T_C$. In particular,
\[
 \Gamma_\tau^2=2-2g(C)=-d.
\]
\end{lemma}
\begin{proof}
Under the parametrization $p\mapsto(p,\tau(p))$, the tangent sequence
of the graph is
\[
 0\longrightarrow T_C
 \xrightarrow{v\mapsto(v,d\tau(v))}
 T_C\oplus\tau^*T_C
 \longrightarrow N_{\Gamma_\tau/(C\times C)}
 \longrightarrow0.
\]
The map $(v,w)\mapsto w-d\tau(v)$ identifies the quotient with
$\tau^*T_C$. The degree of a curve's normal bundle is its
self-intersection.
\end{proof}

\begin{lemma}\label{lem:transversality}
The diagonal $E_1$ and the graph $E_2$ meet transversely at precisely
the $r$ fixed points of $\sigma$. Consequently,
\[
 E_1^2=E_2^2=-d,\qquad E_1\cdot E_2=r.
\]
\end{lemma}
\begin{proof}
A point lies in both graphs exactly when it is of the form $(p,p)$
with $\sigma(p)=p$. At such a point, $(d\sigma_p)^2=1$.
Moreover, $d\sigma_p$ cannot equal $1$: a nontrivial finite-order
holomorphic germ on a smooth complex curve cannot be tangent to the
identity. Indeed, the first nonzero higher-order term of a germ
$z\mapsto z+az^k+\cdots$ would be multiplied by its order upon
iteration, a contradiction in characteristic zero.
Thus $d\sigma_p=-1$.
The tangent lines to the two graphs are therefore generated by
$(v,v)$ and $(v,-v)$; these are distinct. Each local intersection
multiplicity is one. The self-intersections follow from
Lemma~\ref{lem:graph-normal}.
\end{proof}

The exceptional configuration consequently has intersection matrix
\begin{equation}\label{eq:graph-matrix}
 Q=\begin{pmatrix}-d&r\\r&-d\end{pmatrix}.
\end{equation}
Its eigenvectors $(1,1)$ and $(1,-1)$ have eigenvalues $-d+r$ and
$-d-r$, respectively. Both eigenvalues are strictly negative because
$d>r$. Also, $r>0$ makes $E$ connected.

Grauert's criterion states that a compact curve in a smooth complex
surface is exceptional precisely when its intersection matrix is
negative definite; see \cite[p.~367]{Grauert}. Thus, after choosing
a sufficiently small neighborhood $U$ of $E$, there is a proper
holomorphic contraction
\begin{equation}\label{eq:contraction}
 \pi:(U,E)\longrightarrow(X,o)
\end{equation}
to a normal complex surface, with $\pi^{-1}(o)=E$ and with $\pi$ an
isomorphism away from $E$. One takes the normal contraction, equivalently
the one with $\pi_*\OO_U=\OO_X$. Connectedness of $E$ ensures that it is
contracted to a single point. The smoothness of $U$ and the transverse
intersections of $E$ make \eqref{eq:contraction} a good resolution.

The point $o$ is singular and is not a quotient singularity. To see
this, recall that the minimal resolution of a quotient surface
singularity has only rational exceptional curves. A resolution of a
normal surface singularity factors through its minimal resolution by
point blowups; such blowups add only rational curves. In our resolution
both $E_1$ and $E_2$ have genus $g(C)\ge2$. This is impossible for a
quotient singularity, including a smooth point.

\begin{remark}
Two distinct components of a simple normal crossings divisor may meet
at several points. This is allowed in Wahl's definition of a good
resolution. If a convention requiring at most one intersection between
any pair of components is preferred, blowing up the intersection points
produces such a resolution. Wahl's resolution independence leaves
$\cw(X)$ unchanged.
\end{remark}

\section{Computation of the second Chern number}\label{sec:calculation}

Let $p_1,p_2:C\times C\to C$ be the projections and put
\begin{equation}\label{eq:LM-definition}
 L=p_1^*K_C|_U,\qquad M=p_2^*K_C|_U.
\end{equation}
Because $U$ is an open subset of a product of smooth curves, there is
an actual splitting of vector bundles
\begin{equation}\label{eq:cotangent-splitting}
 \Omega_U^1\simeq L\oplus M.
\end{equation}
Every restriction $p_j|_{E_i}$ is an isomorphism. Therefore
\begin{equation}\label{eq:LM-degrees}
 L\cdot E_i=M\cdot E_i=d\qquad(i=1,2).
\end{equation}

\subsection{Solving for the relative classes}

Write $\crel(L)=a_1E_1+a_2E_2$. Equations
\eqref{eq:relative-class} and \eqref{eq:LM-degrees} give
\[
 -da_1+ra_2=d,\qquad ra_1-da_2=d.
\]
Subtracting the two equations yields
$(d+r)(a_2-a_1)=0$, so $a_1=a_2$. Either equation now gives
$a_1=-d/(d-r)$.
The degree vector of $M$ is the same, and hence
\begin{equation}\label{eq:LM-classes}
 \crel(L)=\crel(M)=-\frac{d}{d-r}E.
\end{equation}
The square of the reduced exceptional cycle is
\begin{equation}\label{eq:E-square}
 E^2=E_1^2+2E_1E_2+E_2^2=-2(d-r).
\end{equation}
Lemma~\ref{lem:equal-classes}, or equivalently Wahl's splitting formula,
therefore gives
\begin{equation}\label{eq:ordinary-value}
 \cw(\Omega_U^1)=L\cdot M
 =\frac{d^2}{(d-r)^2}E^2
 =-\frac{2d^2}{d-r}.
\end{equation}

For comparison, the intersection of the extended line bundles on the
compact surface is
\[
 (p_1^*K_C\cdot p_2^*K_C)_{C\times C}=d^2.
\]
This positive global number is not the relative number in
\eqref{eq:ordinary-value}. The latter is obtained by solving against
the exceptional intersection matrix, as in
\eqref{eq:relative-pairing}.

\subsection{The logarithmic correction}

Each $E_i$ has topological Euler characteristic $2-2g(C)=-d$, and the
two components meet at $r$ points. Thus
\begin{equation}\label{eq:Euler-graphs}
 \etop(E)=-2d-r.
\end{equation}
Applying Proposition~\ref{prop:residue-identity} gives
\begin{align}
 \cw(X)
 &=\cw(\Omega_U^1)-\etop(E)\notag\\
 &=-\frac{2d^2}{d-r}+2d+r\notag\\
 &=-\frac{r(d+r)}{d-r}.
 \label{eq:main-calculation}
\end{align}
The last expression is strictly negative. Together with the geometric
construction, this proves Theorem~\ref{thm:main}.

\subsection{The logarithmic first Chern class}

Since
$\det\Omega_U^1(\log E)=\OO_U(K_U+E)$, its relative first Chern class
is the unique rational exceptional cycle $Z$ satisfying
$Z\cdot E_i=(K_U+E)\cdot E_i$ for $i=1,2$. Adjunction and
$E_1\cdot E_2=r$ give
\[
 (K_U+E)\cdot E_i
 =(K_U+E_i)\cdot E_i+(E-E_i)\cdot E_i
 =2g(C)-2+r=d+r.
\]
The two intersection conditions have the same right-hand side, so the
uniqueness of $Z$ and the symmetry exchanging $E_1$ and $E_2$ imply
$Z=aE$ for some $a\in\mathbb Q$. Since
$E\cdot E_i=-d+r=-(d-r)$, one obtains
$a=-(d+r)/(d-r)$. Finally, $E^2=-2(d-r)$, and hence
\begin{equation}\label{eq:log-c1}
 \crel\bigl(\Omega_U^1(\log E)\bigr)
 =-\frac{d+r}{d-r}E,
 \qquad
 \crel\bigl(\Omega_U^1(\log E)\bigr)^2
 =-\frac{2(d+r)^2}{d-r}.
\end{equation}

\section{A filtration realizing the negative value}\label{sec:filtration}

The preceding calculation determines the invariant exactly. We now
construct a filtration that gives the same number directly in the
infimum \eqref{eq:wahl-definition}. This also provides a direct
negative upper bound without using the residue identity.

Put $\GG=\Omega_U^1(\log E)$ and retain $L$ from
\eqref{eq:LM-definition}. The differential of $p_1$ gives an injective
sheaf map
\[
 L\longrightarrow\GG.
\]
Away from the nodes of $E$ this is a line subbundle. Off $E$, this
follows from the smoothness of the projection. At a smooth point of
$E$, one may use coordinates $(z,t)$ with $p_1=z$ and $E=(t=0)$;
the logarithmic frame is $(dz,dt/t)$, in which the map has no zero.

At a fixed point of the involution, choose a coordinate $z$ on $C$
in which $\sigma(z)=-z$. Such a coordinate exists, for example by
replacing any local parameter $t$ with $(t-t\circ\sigma)/2$.
Using coordinates $(z,w)$ on the product, put
\[
 a=w-z,\qquad b=w+z.
\]
Then $E=(ab=0)$ and $p_1=z=(b-a)/2$. In the logarithmic frame
$(da/a,db/b)$,
\begin{equation}\label{eq:local-differential}
 dz=-\frac{a}{2}\frac{da}{a}+\frac{b}{2}\frac{db}{b}.
\end{equation}
The zero ideal is the reduced maximal ideal $(a,b)$.

Let $Z$ be the reduced set of the $r$ nodes, and let
\[
 N=\det\GG\otimes L^{-1}=\OO_U(K_U+E)\otimes L^{-1}.
\]
The map above has the exact sequence
\begin{equation}\label{eq:singular-filtration}
 0\longrightarrow L\longrightarrow\GG
 \longrightarrow I_Z\otimes N\longrightarrow0.
\end{equation}
Indeed, locally at a node its cokernel is the rank-one torsion-free
module given by the ideal $(a,b)$, with determinant $N$.

Blow up all these nodes:
\[
 \beta:\widehat U\longrightarrow U,
\]
and let $D$ be the sum of the $r$ exceptional curves. The pullback of
\eqref{eq:local-differential} has a common factor defining $D$.
After removing that factor, its components have no common zero.
For instance, in the chart $a=t$, $b=tu$, they are proportional to
$(-1,u)$. Thus the saturated pullback of $L$ is a line subbundle and
we obtain an admissible filtration
\begin{equation}\label{eq:resolved-filtration}
 0\longrightarrow\beta^*L\otimes\OO_{\widehat U}(D)
 \longrightarrow\beta^*\GG
 \longrightarrow\beta^*N\otimes\OO_{\widehat U}(-D)
 \longrightarrow0.
\end{equation}

We have $N\cdot E_i=(d+r)-d=r$, so
\[
 \crel(N)=-\frac{r}{d-r}E,
 \qquad L\cdot N=-\frac{2dr}{d-r}.
\]
Since pullback classes have zero intersection with the
$\beta$-exceptional curves and $D^2=-r$, the value supplied by
\eqref{eq:resolved-filtration} is
\begin{align}
 (\beta^*L+D)\cdot(\beta^*N-D)
 &=L\cdot N-D^2\notag\\
 &=-\frac{2dr}{d-r}+r
 =-\frac{r(d+r)}{d-r}.
 \label{eq:filtration-value}
\end{align}
Here addition denotes tensor product with the indicated divisor
line bundle when used inside an intersection expression.
Because $\deg\beta=1$, this is also the normalized value in
\eqref{eq:wahl-definition}. It agrees with the exact value computed
in Section~\ref{sec:calculation}.

\section{An explicit example with second Chern number \texorpdfstring{$-4$}{-4}}
\label{sec:explicit}

Let $C$ be the smooth projective model of
\begin{equation}\label{eq:explicit-curve}
 y^2=(x^2-1)(x^2-2)(x^2-3)(x^2-4)(x^2-5).
\end{equation}
The right-hand side has ten distinct roots. The map to
$\mathbb P^1_x$ is a double cover with ten simple branch points and
no ramification at infinity. Riemann--Hurwitz gives
\[
 2g(C)-2=2(-2)+10=6,
\]
so $g(C)=4$.

The involution
\begin{equation}\label{eq:explicit-involution}
 \sigma(x,y)=(-x,y)
\end{equation}
has exactly two finite fixed points, namely the points over $x=0$
with $y^2=-120$. At either point $y\ne0$, so $x$ is a local parameter
and the derivative of the involution is $-1$.
It remains to check infinity. Put
\[
 u=x^{-1},\qquad v=y/x^5.
\]
Then the equation becomes
\[
 v^2=\prod_{j=1}^5(1-ju^2).
\]
The two points at infinity are $(u,v)=(0,1)$ and $(0,-1)$.
The involution acts by $(u,v)\mapsto(-u,-v)$, so it exchanges these
points. There are no additional fixed points. Thus $d=6$ and $r=2$.

For $E=\Delta_C+\Gamma_\sigma$, the numerical data are
\begin{equation}\label{eq:explicit-matrix}
 Q=\begin{pmatrix}-6&2\\2&-6\end{pmatrix},\qquad
 Q^{-1}=\begin{pmatrix}-3/16&-1/16\\-1/16&-3/16\end{pmatrix}.
\end{equation}
The eigenvalues are $-4$ and $-8$. The relative classes and the Euler
characteristic are
\[
 \crel(L)=\crel(M)=-\frac32E,\qquad
 E^2=-8,\qquad \etop(E)=-14.
\]
Therefore
\begin{equation}\label{eq:explicit-value}
 \cw(\Omega_U^1)=\frac94(-8)=-18,
 \qquad
 \cw(X)=-18-(-14)=-4.
\end{equation}
For this example the logarithmic first Chern class is $-2E$, with
square $-32$.

This exhibits a nonquotient normal surface singularity contradicting
the zero value asserted in \cite[Conjecture~2.12]{Wahl}.

\section{Branched double covers and unbounded negative values}
\label{sec:families}

We next show that the numerical hypotheses of
Theorem~\ref{thm:main} occur in families with arbitrarily large $r$.
Let $B$ be a smooth projective complex curve of genus $h\ge2$.
Choose $2m$ distinct points on $B$, where $m\ge1$, and write their
sum as the reduced divisor $R$.

There is a line bundle $A$ on $B$ with
\begin{equation}\label{eq:square-root}
 A^{\otimes2}\simeq\OO_B(R).
\end{equation}
To see this, choose any line bundle $A_0$ of degree $m$.
Then $\OO_B(R)\otimes A_0^{-2}$ has degree zero. Multiplication by
$2$ on the complex abelian variety $\Pic^0(B)$ is surjective, so
there is $T\in\Pic^0(B)$ satisfying
$T^{\otimes2}\simeq\OO_B(R)\otimes A_0^{-2}$. Taking
$A=A_0\otimes T$ proves \eqref{eq:square-root}.

Let $s\in H^0(B,A^{\otimes2})$ be the section with divisor $R$.
The corresponding double cover is
\begin{equation}\label{eq:double-cover}
 q:C=\Spec_B\bigl(\OO_B\oplus A^{-1}\bigr)\longrightarrow B,
\end{equation}
where multiplication $A^{-1}\otimes A^{-1}\to\OO_B$ is defined by
$s$. Locally this is the equation $w^2=s(z)$.
At a zero of $s$, that zero is simple, and the cover is smooth with
local form $w^2=z$. Away from $R$ it is \'{e}tale. The cover is connected:
if it were disconnected, its two components would each map with
degree one to the normal curve $B$, and hence would be isomorphic to
$B$, which would leave no branch points. Thus $C$ is a smooth
connected projective curve.

The deck involution fixes exactly the $r=2m$ ramification points.
Riemann--Hurwitz gives
\begin{equation}\label{eq:RH-family}
 d=2g(C)-2=2(2h-2)+r=4h-4+r.
\end{equation}
In particular $d-r=4h-4>0$, so the graph configuration is negative
definite for every $m\ge1$.

\begin{corollary}\label{cor:double-cover-values}
The contraction associated with the double cover
\eqref{eq:double-cover} satisfies
\begin{equation}\label{eq:double-cover-value}
 \cw(X)=-r-\frac{r^2}{2h-2}
 =-2m-\frac{2m^2}{h-1}.
\end{equation}
\end{corollary}
\begin{proof}
Substituting \eqref{eq:RH-family} into
\eqref{eq:main-formula} yields
\[
 -\frac{r(4h-4+2r)}{4h-4}
 =-r-\frac{r^2}{2h-2}.
\]
Now use $r=2m$.
\end{proof}

Taking $h=2$ gives $g(C)=m+3$ and
\[
 \cw(X_m)=-2m(m+1),
\]
proving Corollary~\ref{cor:unbounded-intro}. The genera of the exceptional components
and their intersection numbers vary with $m$; no claim of unboundedness
for a fixed resolution graph is being made.

\section{Relation with the asymptotic Riemann--Roch formula}
\label{sec:rr}

For a vector bundle $\FF$ on a good resolution, Wahl uses the modified
Euler characteristic
\begin{equation}\label{eq:modified-euler}
 \chi'(\FF)=
 \dim_{\mathbb C}\frac{H^0(U\setminus E,\FF)}{H^0(U,\FF)}
 +\dim_{\mathbb C}H^1(U,\FF).
\end{equation}
The asymptotic formula in
\cite[Main Conjecture~3.3]{Wahl} is
\begin{equation}\label{eq:RR-formula}
 \chi'(\Sym^n\FF)
 =-\frac{n^3}{6}\bigl(\crel(\FF)^2-\cw(\FF)\bigr)+O(n^2).
\end{equation}
Its validity is a different issue from vanishing of
$\cw(\Omega_U^1(\log E))$.

\begin{proposition}\label{prop:RR-specific}
For every singularity constructed in Theorem~\ref{thm:main},
formula \eqref{eq:RR-formula} holds for both
$\FF=\Omega_U^1$ and $\FF=\Omega_U^1(\log E)$.
More explicitly,
\begin{align}
 \chi'(\Sym^n\Omega_U^1)
 &=\frac{d^2}{d-r}n^3+O(n^2),
 \label{eq:RR-ordinary}\\
 \chi'\bigl(\Sym^n\Omega_U^1(\log E)\bigr)
 &=\frac{(d+r)(2d+r)}{6(d-r)}n^3+O(n^2).
 \label{eq:RR-log}
\end{align}
\end{proposition}
\begin{proof}
Wahl proves the formula for direct sums of line bundles
\cite[Corollary~3.6]{Wahl}, so it holds for
$\Omega_U^1=L\oplus M$.
Wahl also proves that two rank-two bundles on the same resolution,
isomorphic away from $E$, either both satisfy the formula or both
fail it \cite[Corollary~3.9]{Wahl}.
The ordinary and logarithmic cotangent bundles agree on $U\setminus E$,
so the formula holds for the latter as well.

For the coefficients, \eqref{eq:LM-classes} and
\eqref{eq:E-square} give
\[
 \crel(\Omega_U^1)^2=-\frac{8d^2}{d-r},
 \qquad \cw(\Omega_U^1)=-\frac{2d^2}{d-r}.
\]
Substitution into \eqref{eq:RR-formula} gives
\eqref{eq:RR-ordinary}.
For the logarithmic bundle, \eqref{eq:log-c1} and
\eqref{eq:main-formula} yield
\[
 \crel\bigl(\Omega_U^1(\log E)\bigr)^2-\cw(X)
 =-\frac{(d+r)(2d+r)}{d-r},
\]
which proves \eqref{eq:RR-log}.
\end{proof}

In the explicit genus-four example, the two leading coefficients are
$9$ and $14/3$, respectively.
Thus the negative logarithmic second Chern number occurs even for a
bundle whose asymptotic Riemann--Roch formula is already covered by
Wahl's theorems. Proposition~\ref{prop:RR-specific} makes no statement
about arbitrary vector bundles on these singularities or about the
general asymptotic conjecture.

\section*{Acknowledgements}

The author used DeepSeek for language polishing, grammar correction, and structural revision during the preparation of this manuscript. The author reviewed and edited all AI-generated output and takes full responsibility for the content of the publication.

\begingroup
\footnotesize
\noindent
School of Mathematical Sciences, Key Laboratory of MEA (Ministry of Education)
and Shanghai Key Laboratory of PMMP, East China Normal University,
Rm.~312, Math. Bldg, No.~500, Dongchuan Road, Shanghai, 200241, P.~R. China.
\quad E-mail address: \texttt{rdu@math.ecnu.edu.cn}
\endgroup
\end{document}